\documentclass[11pt]{article}
\usepackage{amsfonts,amsthm,amsmath,hyperref}
\begin{document}
\newtheorem{thm}{Theorem}[section]
\newtheorem{lem}[thm]{Lemma}
\newtheorem{cor}[thm]{Corollary}
\theoremstyle{definition}
\newtheorem{defn}[thm]{Definition}
\newtheorem{conj}{Conjecture}
\newtheorem{quest}[conj]{Question}
\theoremstyle{remark}
\newtheorem*{remark}{Remarks}
\newcommand{\beq}{\begin{equation}}
\newcommand{\eeq}{\end{equation}}
\title{Denominators for One Variable Poincar\'e
Series of Generic Matrices}
\author{Allan Berele\\ Department of Mathematical
Sciences\\ DePaul University\\ Chicago, IL 60614\\
USA}
\abstract{We study the Poincar\'e series of the
 mixed and pure trace rings of generic matrices.  These series
are known to be rational functions.  We obtain an explicit 
formula in lowest terms in the case of $2\times2$ matrices; a denominator,
which we presume but have not been able
to prove to be in lowest terms,
in the case of $3\times3$ matrices; and a conjectured denominator
in the general case.}
\maketitle
\section{Introduction}
This paper is concerned with the Poincar\'e series
of two families of algebras, $\bar{C}=\bar{C}(n,k)$ and
$\bar{R}=\bar{R}(n,k)$, which in p.i. theory are called the
pure and mixed trace rings of generic matrices, and
in invariant theory are called the algebras of matrix
invariants and concomitants.

Let $X_\alpha$ be the $n\times n$ matrix with each $(i,j)$ entry equal to the commuting indeterminate $x_{ij}^{(\alpha)}$, and let
$F$ be a field of characteristic zero.  The algebra $F[X_1,\ldots,X_k]$, which 
is contained in the algebra of $n\times n$ matrices over the 
polynomial ring $F[x_{ij}^{\alpha)}]_{i,j,\alpha}$, is called the algebra of
generic matrices and it satisfies various universal properties which we do not 
discuss here.  There are three additional algebras that may be constructed from
$R$.  One is $C$, the center of $R$.  The other two are constructed using the
trace map $tr:R\rightarrow F[x_{ij}^{\alpha)}]_{i,j,\alpha}$.  The algebra $\bar{C}
=\bar{C}(n,k)$ is called the pure trace ring and
is defined to be the commutative algebra generated by the image of $tr$; and
$\bar{R}=\bar{R}(n,k)$ is called the mixed
trace ring, and is defined to be the algebra generated by $R$ and $\bar{C}$.

The two algebras $\bar{C}$ and $\bar{R}$, which will be the main focus of 
this work, can also be defined using invariant theory.  A function $f:M_n(F)^k
\rightarrow F$ is said to be invariant if 
$$f(x_1,\dots,x_k)=f(gx_1g^{-1},\ldots,gx_kg^{-1})$$
for all $x_1,\ldots,x_k\in M_n(F)$ and all $g\in GL_n(F)$; and it is said to
be polynomial in the entries if each entry of $f(x_1,\ldots,x_k)$ is a polynomial
function of the entries of $x_1,\ldots,x_k$.  An important example would be
$tr(u)$ where $u$ is any polynomial in $x_1,\ldots,x_k$.  The First Fundamental
Theorem of invariant theory says that the algebra of invariant polynomial functions
on $M_n(F)$ is generated by these trace functions, $tr(u)$.  Using this it can be proven 
that this algebra is isomorphic to $\bar{C}$, defined above.  A similar
construction can be used to show that if one considers
 functions from
$M_n(F)^k$ to $M_n(F)$, instead of to $F$, invariant and polynomial in the
entries, that the algebra of such functions is isomorphic to $\bar{R}$.  Obviously, this is a bare-bones description of a deep subject.  If the reader
wishes to learn more there are many resources including~\cite{AGPR}
and~\cite{F}.

Each of the four rings has an $\mathbb{N}$ grading by total degree and a
finer $\mathbb{N}^k$ grading by degree in the individual matrices.  Correspondingly, there are eight Poincar\'e series one might study:
$P(R),$ $P(C),$ $P(\bar{R})$, $P(\bar{C})$, $\tilde{P}(R),$ $\tilde{P}(C),$
 $\tilde{P}(\bar{R})$, and $\tilde{P}(\bar{C})$, where $P$ is used for the
 series in one variable and $\tilde{P}$ is used for the $k$~variable series.
Then,
$$\tilde{P}(X)(t,t,\ldots,t)=P(X)$$
for $X$ any one of $C$, $R$, $\bar{C}$ or $\bar{R}$.

 These series have been the focus of a lot of research.  One important
 theorem is that each is a rational function whose denominator can be
 taken to be a product of terms of the form $(1-t^u)$, where $t^u$ is a monomial
of degree at most~$n$. In \cite{F} Formanek 
described how these series could be computed using
complex integration, but it was Teranishi who
first made use of this suggestion in~\cite{T}, \cite{T87}
and~\cite{T88}.  The Poincar\'e series
for $\bar{C}(n,k)$ equals
\beq(2\pi i)^{-n}(n!)^{-1}\oint_T
\frac{\prod_{i\ne j=1}^n(1-\frac{z_i}{z_j})}
{\prod_{\alpha=1}^k\prod_{i,j=1}^n(1-\frac{z_i}{z_j}
t_\alpha)}
\frac{dz_1}{z_1}\wedge\cdots\wedge\frac{dz_n}{z_n}\eeq
where $T$ represents the torus $|z_1|=\cdots=|z_n|=1$.
The Poincar\'e series for $\bar{R}(n,k)$ is the same with
an extra factor of $\sum\frac{z_i}{z_j}=\sum z_i\sum
z_j^{-1}$ in the numerator:
\beq(2\pi i)^{-n}(n!)^{-1}\oint_T
\frac{\sum_{i,j=1}^n z_iz_j^{-1}\prod_{i\ne j=1}^n(1-\frac{z_i}{z_j})}
{\prod_{\alpha=1}^k\prod_{i,j=1}^n(1-\frac{z_i}{z_j}
t_\alpha)}
\frac{dz_1}{z_1}\wedge\cdots\wedge\frac{dz_n}{z_n}\eeq
These integrals were used to study the Poincar\'e
series of $\bar{C}(n,k)$ and $\bar{R}(n,k)$ in a 
number of papers including \cite{B05}, \cite{BS},
\cite{D}, \cite{T}, \cite{T87}, \cite{T88} and \cite{V}.

By setting all of the $t_\alpha$ to a single variable
$t$ in~(1) and~(2) we get the the one variable Poincar\'e series for
$\bar{C}(n,k)$ and $\bar{R}(n,k)$.  The former would
be
\beq(2\pi i)^{-n}(n!)^{-1}\oint_T
\frac{\prod_{i\ne j=1}^n(1-\frac{z_i}{z_j})}
{\prod_{i,j=1}^n(1-\frac{z_i}{z_j}
t)^k}
\frac{dz_1}{z_1}\wedge\cdots\wedge\frac{dz_n}{z_n}\label{eq:3a}\eeq
and the latter
\beq(2\pi i)^{-n}(n!)^{-1}\oint_T
\frac{\sum_{i,j=1}^n z_iz_j^{-1}\prod_{i\ne j=1}^n(1-\frac{z_i}{z_j})}
{\prod_{i,j=1}^n(1-\frac{z_i}{z_j}
t)^k}
\frac{dz_1}{z_1}\wedge\cdots\wedge\frac{dz_n}{z_n}\eeq

 In the case of $n=2$, namely
$2\times2$ matrices, the series have been computed in various ways, see \cite{F}
and \cite{L}.   An explicit formula for the one variable Poincar\'e series $P(
\bar{C}(2,k))$
was proven by Teranishi in \cite{T}:
$$P(\bar{C}(2,k))=(-1)^{k-1}\frac1{2(k-1)!(1-t)^{2k}}\left(\frac d{dz}
\right)^{k-1}\left.\frac{z^{k-2}(z-1)^2}{(tz-1)^k}\right|_{z=t}$$
In the next two sections of this paper we will prove a different formula
for $P(\bar{C}(2,k))$ and also one for $P(\bar{R}(2,k))$ expressing
each as a rational function in lowest terms.
  Theorem~\ref{thm:3.1} states
that $\bar{C}(2,k)$ has Poincar\'e series
$$(1-t)^{2-2k}(1-t^2)^{1-2k}
\left(\sum{k-2\choose i}^2t^{2i}-\sum{k-2\choose i}{k-2
\choose i+1}t^{2i+1}\right)$$
and Theorem~\ref{thm:3.2} states that
 the Poincar\'e series of
$\bar{R}(2,k)$ equals
$$(1-t)^{-2k}(1-t^2)^{3-2k}\sum\frac1{k-2}
{k-2\choose i}{k-2\choose i+1}x^{2i}$$
  The coefficients in the numerator of $P(\bar{R}(2,k))$ 
 turn out to be the Narayama
numbers.  For more on these numbers, see~\cite{P}.

 In \cite{BS} Berele and Stembridge found denominators for
$\bar{P}(\bar{C}(n,k))$ and $\bar{P}(\bar{R}(n,k))$ for $n\le 4$ and showed that
they were least denominators for $n\le 3$.  In the current
paper we will be studying the denominators for the
one variable Poincar\'e series. Of course, one could find 
denominators for the one variable Poincar\'e series
by simply specializing that of the $k$ variable series, however
such a denominator would be far from least.  For example,
$\bar{P}(\bar{C}(2,k)$ has denominator 
$\prod_i(1-t_i)\prod_{i\le j}(1-t_it_j)$ which
specializes to $(1-t)^k(1-t^2)^{k+1\choose 2}$, but
the least denominator for $\bar{P}(\bar{C}(2,k)$
 is $(1-t)^{2k-2}(1-t^2)^{2k-1}$.
  Based on the 
$k$ variable case we know that the denominators for the
one variable Poincar\'e series of $\bar{C}(n,k)$ and $\bar{R}(n,k)$
are products of terms $(1-t^i)$ with $i\le n$. For $n=3$ we
find denominators which we believe to be least, although
we do not prove it. In the case of $P(\bar{C}(3,k))$ the denominator
is $(1-t)^{2k-2}(1-t^2)^{4k-4}(1-t^3)^{3k-2}$, and in the case of
$P(\bar{R}(3,k))$ the denominator is $(1-t)^{2k}(1-t^2)^{4k-4}(1-t^3)^{3k-4}$.
  Zeilberger has computed $P(\bar{C}(3,k))$ 
for $k\le 100$, see~\cite{Z}, and in each of these cases our denominators
are least. At any rate, the degrees are again much
smaller than those gotten by specializing the $k$ variable
series. Finally, in Section~5 we present a conjecture for the
denominators of all ${P}(\bar{C}(n,k))$ and ${P}(\bar{R}(n,k))$
which agrees with our results for $n=2,3$ and with
the denominators computed by Dokovi\'c in \cite{D} for
$n=4,5,6$ and $k=2$.  Since these are all the known denominators at this point we feel confident that our conjecture is correct.

 Procesi proved
that $\bar{C}(n,k)$ and $\bar{R}(n,k)$ each have Gelfand-Kirillov dimension equal to
$(k-1)n^2+1$ implying this lemma.
\begin{lem} The Poincar\'e series for each $\bar{C}(n,k)$ and
$\bar{R}(n,k)$ has a pole at $t=1$ of order $(k-1)n^2+1$.\label{lem:1.1}
\end{lem}
This lemma will be useful in our computation of the denominators partly
because if we can compute the multiplicity of each $(1-t^i)$ in
the denominator for $i\ge2$, then we can deduce the multiplicity of $(1-t)$.

Here is another useful
theorem due to Formanek in \cite{F86}, originally stated for the multiple
variable Poincar\'e series.
\begin{lem} If $F(t)$ equals either $P(\bar{C}(n,k))$ or
$P(\bar{R}(n,k))$, then $F(t)$ satisfies the functional equation
 $F(\frac1t)=\pm t^{kn^2}F(t)$.  In particular, if the 
denominator has degree $d$ then the numerator will have
degree $d-kn^2$.\label{lem:1.2}
\end{lem}

The current work was inspired by a much larger question.  Given a p.i. algebra
$A$, perhaps with~1, there are various integer sequences $c_m(A)$ related to the cocharacter
of $A$ which are known to be asymptotic to some $\alpha m^d$.  In this
paper we are focusing on the case in which $c_m(A)$ equals the growth function for a generic algebra for
$A$.  Other sequences known to be asymptotic to a polynomial include
the colength sequence and the maximal multiplicity sequence, see~\cite{B06}
and~\cite{B08}.  In the case of verbally prime algebras and prime product
algebras the exponent~$d$ has been computed
for the growth functions of the generic algebras
and for the colegnth and maximal multiplicity
sequences, see~\cite{B93}, \cite{B05}, \cite{B08a},
\cite{B10} and~\cite{B10a}, but the coefficients $\alpha$ can only be computed
in those rare cases in which the codimension sequence is completely known.  We feel that these coefficients should be of interest.

Generally, if $f(t)$ and $g(t)$ are polynomials not divisible by $(1-t)$ and if
all of the zeros of $g(t)$ are on the unit circle, each with multiplicity less
than~$d$, then if
$$\frac{f(t)}{(1-t)^dg(t)}=\sum c_mt^m$$
the coefficient $c_m$ will be asymptotic to
$\frac{f(1)}{g(1)}{m+d-1\choose d-1}$ or
 $\frac{f(1)}{g(1)(d-1)!}m^{d-1}$.
It follows that if $P(\bar{C}(k,n))=\sum \bar{c}_mt^m$ and $P(\bar{r}(n,k))
=\sum \bar{r}_mt^m$ then each of $\bar{c}_m$ and $\bar{r}_m$ is
asymptotic to some $\alpha m^{(k-1)n^2}$ where $\alpha$ is a rational
number with denominator of the form $\alpha'[(k-1)n^2]!$, where
$\alpha'$ has all prime factors less than or equal to~$n$.  If our conjecture
is true, one could identify the denominator more precisely.  Likewise, if
$\gamma_m=\sum_{i=0}^m \bar{c}_i$ and $\rho_m=\sum_{i=0}^m
\bar{r}_i$, then $\sum \gamma_mt^m=(1-t)^{-1}\sum \bar{c}_mt^m$
and  $\sum \rho_mt^m=(1-t)^{-1}\sum \bar{r}_mt^m$, and so each of
$\gamma_m$ and $\rho_m$ is asymptotic to a constant times 
$m^{(k-1)n^2+1}$ where the constant factor is the same rational number
as before, divided by $(k-1)n^2+1$.

In the case of $P(\bar{C}(2,k))=N_k(t)(1-t)^{2-2k}(1-t^2)^{1-2k}$
Teranishi computed $N_k(1)=\frac1{k-1}{2k-4\choose k-2}$. These
are the Catalan numbers $C_{k-2}$, and
so $\bar{c}_m\simeq C_{k-2}2^{1-2k}{m+4k-2\choose 4k-2}.$    It follows from our work that if
$P(\bar{R}(2,k))=N_k(t)(1-t)^{2k}(1-t^2)^{3-2k}$ then
$N_k(1)$ also equals the Catalan number $C_{k-2}$ and so
$\bar{c}_m\simeq C_{k-2}2^{3-2k}{m+4k-2\choose 4k-2}$.

In the case of $P(\bar{C}(3,k))=N_k(t)(1-t)^{2-2k}(1-t^2)^{4-4k}(1-t^3)^{2-3k}$, so $\bar{c}_m$ is asymptotic to $N_k(1)2^{4-4k}3^{2-3k}
{m+9k-8\choose 9k-7}$.
It is possible to compute $N_k(1)$ for low values
of $k$ using Zeilberger's results.  Sadly, no pattern is apparent.
The first few are $N_2(1)=1,$ $N_3(1)=21$ and $N_4(1)=1636$.

We are happy to acknowledge B. Tenner's help in using oeis; useful conversations
with S. Catoiu and K. Liechty; and D. Zeilberger for patiently explaining how to
use his algorithm in Maple and, as already mentioned, computing the Poincar\'e
series $P(\bar{C}(2,k))$ and $P(\bar{C}(3,k))$ for $k\le 100$ in~\cite{Z}.  These computations involved  new recurrence
relations he found for these functions.  The relation
for $n=2$ is fairly simple.  If $B(k)=(1-t)^{2k}
P(\bar{C}(2,k)$ then $B(k)$ equals $(1-t^2)^{-2}
(k-1)^{-1}$ times
$$2(kt^2-2t^2+k-t-2)B(k-1)-(k-3)B(k-2).$$
On the other hand, the relation for $n=3$
is anything but simple and takes up hundreds
of lines of output!

\section{Two-By-Two Matrices - Denominators}
In \cite{T}, a work which served as inspiration for much
of this paper, Teranishi computed the asymptotics of the coefficients in Poincar\'e series in
the case of $2\times2$ matrices.  We now 
adapt his methods to give a more explicit
result.

The Poincar\'e series for the trace ring $\bar{C}(2,k)$ of
$2\times 2$  generic matrices for $k\ge2$ is given by the integral
\begin{equation}\label{eq:1a}
\frac1{2!}(2\pi i)^{-2}\oint_{|z_2|=1}\oint_{|z_1|=1}
\frac{(1-\frac{z_1}{z_2}t)(1-\frac{z_2}{z_1})}{(1-t)^{2k}
(1-t\frac{z_1}{z_2})^k(1-t\frac{z_2}{z_1})^k}\frac{dz_1}{z_1}
\wedge\frac{dz_2}{z_2}\end{equation}
 Pulling out the factor of $(1-t)^{2k}$ and clearing fractions
we get  $-\frac12(2\pi i)^{-2}(1-t)^{-2k}$ times
$$\oint_{|z_2|=1}\oint_{|z_1|=1}\frac{-z_1^{k-2}z_2^{k-2}
(z_1-z_2)^2}{(z_1-z_2t)^k(z_2-z_1t)^k} dz_1\wedge dz_2.$$
Since $t$ is taken to be less than~1 the only pole in the unit
circle is at $z_1=tz_2$ and it is of order~$k$.  By Cauchy's
integration formula we can now evaluate $(2\pi i)^{-1}$ times
 the inner integral
in three steps:  (1)~multiply by $(z_1-tz_2)^k/(k-1)!$  This step will have the effect
of cancelling the $(z_1-tz_2)^k$ term from the denominator.
(2)~Take the $(k-1)^{st}$ derivative with respect to~$z_1$ and
(3)~take the limit as $z_1\rightarrow tz_2$.  Letting $D$
denote the derivative with respect to $z_1$, the multiplication by $(z_1-tz_2)^{k-1}$
derivative $D^{k-1}$ transforms the integrand to 
\begin{equation}
(k-1)!^{-1}\sum_{a+b+c=k-1}{k-1\choose a,b,c}-D^{a}(z_1^{k-2})
z_2^{k-2}D^b(z_1-z_2)^2D^c(z_2-z_1t)^{-k}.
\label{eq:1}\end{equation}
At this point it is useful to use the notation of rising and
falling factorials.  $(n)_a$ is defined to be $n(n-1)\cdots
(n-a+1)$ or $n!/(n-a)!$; and $n^{(a)}$ is defined to be
$n(n+1)\cdots (n+a-1)$ or $(n+a-1)!/(n-1)!$.  Taking into account that the $(k-1)!$ coefficient
will cancel the $(k-1)!$ from the binomial coefficient, the
summand in (\ref{eq:1}) corresponding to a given $a,b,c$
equals
$$-\frac1{a!b!c!}(k-2)_a z_1^{k-2-a}z_2^{k-2}
(2)_b(z_1-z_2)^{2-b} k^{(c)}t^c(z_2-tz_1)^{-k-c}.$$
Substituting $z_1=tz_2$ then gives
$$-\frac1{a!b!c!} (k-2)_a(2)_bk^{(c)}t^{k-2-a}
z_2^{k-2-a}z_2^{k-2}(t-1)^{2-b}t^cz_2^{-k-c}
(1-t^2)^{-k-c}$$
Collecting the powers of $z_2$ gives $z_2$ to the
power of 
$$k-2-a+k-2+2-b-k-c=k-2-(a+b+c)=k-2-(k-1)=-1.$$
The integral with respect to $z_2$ will merely 
cancel that term and the remaining $(2\pi i)^{-1}$
in front, giving this theorem:
 \begin{thm} The Poincar\'e series of $\bar{C}$ for $k$
$2\times 2$ matrices equals the sum over all $a+b+c=
k-1$, $b\le2$ of $(1-t)^{-2k}$ times
$$-\frac1{2\cdot a!b!c!} (k-2)_a(2)_bk^{(c)}t^{k-2-a+c}
(t-1)^{2-b}
(1-t^2)^{-k-c}$$\label{th:2.1}
\end{thm}
In order to identify the least denominator, note that the
order of the pole at $t=1$ in each term is $2k-2+b+k+c$.
This is maximal when $b+c=k-1$ and so is $4k-3$, in
agreement with Lemma~\ref{lem:1.1}.  And
the order of the pole at $t=-1$ in each term is $k+c$ with unique maximum
at $2k-1$.
\begin{cor} The Poincar\'e series of $\bar{C}$ for $k$
$2\times 2$ matrices is a rational function with 
denominator $(1-t)^{2k-2}(1-t^2)^{2k-1}$
\end{cor}
We record the first few numerators, which we denote $N_k$:
\begin{align*}
N_2&=1\\
N_3&=t^2-t+1\\
N_4&=t^4-2t^3+4t^2-2t+1\\
N_5&=t^6-3t^5+9t^4-9t^3+9t^2-3t+1\\
N_6  &=t^8-4t^7+16t^6-24t^5+36t^4-24t^3+15t^2-4t+1\\
N_7&=t^{10}-5t^9+25t^8-50t^7+100t^6-100t^5
+100t^4-50t^3+25t^2-5t+1\\
N_8&=t^{12}-6t^{11}+36t^{10}-90t^9+225t^8
-300t^7+400t^6-300t^5+225t^4-90t^3\\ &\quad+36t^2
-6t+1
\end{align*}

With some help from oeis we found that the coefficients in $N_k$ alternate between ${{n-2}\choose m}^2$
and $-{{n-2}\choose m}{{n-2}\choose m+1}$.  To see
what this means consider, for example, the fourth row of
Pascal's triangle:
$$1,\quad 4,\quad6,\quad4,\quad 1$$
Then the coefficients in $N_6$, up to sign, equal
$$1\cdot1,\quad1\cdot4,\quad4\cdot4,\quad4\cdot6
,\quad6\cdot6,\quad6\cdot4,\quad4\cdot4,\quad4\cdot
1,\quad1\cdot1$$
Before proving this in the next section, we turn
to the case of $\bar{R}$.

In order to compute the Poincar\'e series of $\bar{R}$
the computation is almost identical, except that the integrand
has an extra factor of $(2+\frac{z_1}{z_2}+\frac{z_2}{z_1})$
 in the numerator. Since the denominator will follow from 
the computations in the next section we leave the proof of
the following to the interested reader.
\begin{thm} The Poincar\'e series of $\bar{R}$ for 
$k$ $2\times 2$ matrices is a rational function with
 denominator $(1-t)^{2k}(1-t^2)^{2k-3}$.
\label{cor:2.3}\end{thm}
Again denoting the numerator as $N_k$ we record:
\begin{align*}
N_2&=1\\
N_3&=1\\
N_4&=t^2+1\\
N_5&=t^4+3t^2+1\\
N_6&=t^6+6t^4+6t^2+1\\
N_7&=t^8+10t^6+20t^4+10t^2+1\\
N_8&=t^{10}+15t^8+50t^6+50t^4+15t^2+1\\
N_9&=t^{12}+21t^{10}+105t^8+175t^6+105t^4+
21t^2+1\\
N_{10}&=t^{14}+28t^{12}+196t^{10}+490t^8+490t^6
+196t^4+28t^2+1
\end{align*}
With some help from oeis we found that these are Narayana
numbers.  The coefficient of $t^{2k}$ in $N_n$ is
equal to 
$$\frac1{n-2}{n-2\choose k}{n-2\choose k+1}.$$
\section{Two-By-Two Matrices - Numerator}
Following Teranishi, we can write the Poincar\'e
seris for $\bar{C}(2,k)$ and $\bar{R}(2,k)$
as integrals over one complex variable instead of two.  
$\bar{C}(2,k)$ has one variable Poincar\'e series 
equal to the integral
\beq \tfrac12 (1-t)^{-2k}(2\pi i)^{-1}\oint_{|z|=1}\frac{
(1-z)(1-z^{-1})}{(1-zt)^k(1-t/z)^k}\:\frac{dz}z
\label{eq:3}\eeq
and $\bar{R}(2,k)$ has Poincar\'e series 
\beq \tfrac12(1-t)^{-2k}(2\pi i)^{-1}\oint_{|z|=1}\frac{
(1-z^2)(1-z^{-2})}{(1-zt)^k(1-t/z)^k}\:\frac{dz}z,
\label{eq:4}\eeq
where $|t|<1$.  By Cauchy
s theorem, in order to evaluate an itegral of the
form $(2\pi i)^{-1}\oint_{|z|=1} f(z)\frac{dz}z$ we
can expand $f(z)$ as a Taylor series in $z$ and take the
constant coefficient.

Expanding equation (\ref{eq:3}) as power
series in $t$ we get $\frac12(1-t)^{2k}$ times the integral of
$$(1-z)(1-z^{-1})
\sum_{a=0}^\infty {{a+k-1}\choose k-1}t^az^a
\sum_{b=0}^\infty {{b+k-1}\choose k-1}t^bz^{-b}$$
times $\frac{dz}z$ which picks out the constant terms in $z$. The product of
the first two terms is $2-z-z^{-1}
$, so in the product of the two summations
we need only consider the terms with $a=b$ or
with $a=b\pm 1$.  For the former the $\frac12$ and
2 cancel and we get
$$\sum_{a=0}^\infty{{a+k-1}\choose k-1}^2 t^{2k}.$$
For the latter we get twice (also cancelling the $\frac12$) the sum
$$\sum_{a=1}^\infty {a+k-1\choose k-1}{a+k-2\choose k-2}t^{2k-1}$$
At this point we wish to prove two combinatorial identities:
\begin{equation}
\sum_a{a+k-1\choose k-1}^2t^{2a}=(1-t^2)^{1-2k}\sum_a {k-1\choose a}^2t^{2a},\label{eq:5}
\eeq
\beq 
\sum_a {a+k\choose k-1}{a+k-1\choose k-1}t^{2a+1}=
(1-t^2)^{1-2k}t\sum_a {k-2\choose a}{k\choose a+1}t^{2a}.\label{eq:6}
\eeq
These can now be proven with software instead of erudition
and intelligence.  To prove (\ref{eq:5}) we run the following
Maple commands (kindly supplied to us by Doron
Zeilberger):

\begin{verbatim}
Z:=SumTools[Hypergeometric][Zeilberger];
ope1:=Z(binomial(a+k-1,k-1)^2*t^a,a,n,K)[1];
\end{verbatim}

The program returns the output
$$ope1:=(kt^2-2kt+t^2+k-2t+1)K^2+(-2kt-2k-t-1)K+k$$
Which is Maple's way of saying that if $F(k)=\sum_a{a+k-1\choose a}^2t^a$ then
$$(kt^2-2kt+t^2+k-2t+1)F(k+2)+(-2kt-2k-t-1)F(k+1)+k
F(k)=0$$
or
$$F(k+2)=\frac{2kt+2k+t+1}{kt^2-2kt+t^2+k-2t+1}F(k+1)
-\frac k{kt^2-2kt+t^2+k-2t+1}F(k)$$
for all $k\ge0$.
If we then run the same program on 
\begin{verbatim}
ope2:=Z((1-t)^(1-2*k)*binomial(k-1,a)^2*t^a,k,a,k)[1];
\end{verbatim}
we get that $G(k)=\sum_n=(1-t)^{1-2k}\sum_a{k-1
\choose a}^2t^a$ satisfies the exact same recurrence.
Since $F(0)=G(0)=0$ and $F(1)=G(1)=(1-t)^{-1}$ this
proves the $G(k)=F(k)$ for all $k\ge 0$ and so proves
(\ref{eq:5}).

The proof of (\ref{eq:6}) is similar.  If we now let 
$F(k)$ equal the right hand side of (\ref{eq:6}) and $G(k)$
equal the left  then each satisfies
\begin{multline*}(k^2+k)X(k+2)+(-2k^2t^2-kt^2-2k^2-k)X(k+1)+\\
(k^2t^4-2k^2t^2-t^4+k^2+2t^2-1)X(k)=0
\end{multline*}

It now follows that $\bar{C}(2,k)$ has Poincar\'e series
\beq(1-t)^{-2k}(1-t^2)^{1-2k}\left(\sum{k-1\choose i}^2t^{2i}-\sum{k\choose i+1}
{k-2\choose i}t^{2i+1}\right)\label{eq:7}\eeq
It turns out that this is not in lowest terms and one more
simplification is possible.
\begin{thm}\label{thm:3.1}
$\bar{C}(2,k)$ has Poincar\'e series
$$(1-t)^{2-2k}(1-t^2)^{1-2k}
\left(\sum{k-2\choose i}^2t^{2i}-\sum{k-2\choose i}{k-2
\choose i+1}t^{2i+1}\right)$$
\end{thm}
\begin{proof}We need to show that
\begin{align*}\sum&{k-1\choose i}^2t^{2i}-\sum{k\choose i+1}{k-2
\choose i}t^{2i+1}=\\&(1-t)^2\left(
\sum{k-2\choose i}^2t^{2i}-\sum{k-2\choose i}{k-2
\choose i+1}t^{2i+1}\right)\end{align*}
Expanding $(1-t)^2=1-2t+t^2$ and first comparing coefficients
of $t^{2i}$ on both sides, we see that we need
$${k-1\choose i}^2={k-2\choose i}^2+2{k-2\choose i-1}
{k-2\choose i}+{k-2\choose i-i}^2.$$
This is simply the square of ${k-1\choose i}={k-2\choose i}
+{k-2\choose i-1}$.  Finally, comparing the coefficients of
$t^{2i+1}$ we need to prove
$${k\choose i+1}{k-2\choose i}={k-2\choose i}{k-2
\choose i+1}+{k-2\choose i-1}{k-2\choose i}+2{k-2
\choose i}^2$$
Dividing both sides by $k-1\choose i$ the left hand
side becomes $k\choose i+1$ and the right side becomes
\begin{align*}
{k-2\choose i+1}+&{k-1\choose i-1}+2{k-1\choose i}=
\\ &{k-2\choose i+1}+{k-1\choose i} +{k-1\choose i-1}+{k-1\choose i}=\\
&{k-1\choose i+1}+{k-1\choose i}=\\
&{k\choose i+1}
\end{align*}

\end{proof}

\begin{remark} We remark that this fraction is in lowest
terms.  By Lemma~\ref{lem:1.1} the order
of the pole at $t=1$ is $4k-3$, which equals
$(2k-2)+(2k-1)$, so no factor of $(1-t)$ can
cancel.  Moreover, the numerator cannot have
a factor of $(1+t)$ else it would have a root
at $t=-1$, which it clearly does not.
\end{remark}

The case of $\bar{R}(2,k)$ is similar and
slightly easier.  The integral (\ref{eq:4})
equals
$$ \tfrac12 (1-t)^{-2k}(2\pi i)^{-1}\oint_{|z|=1}\frac{
2-z^2-z^{-2}}{(1-zt)^k(1-t/z)^k}\:\frac{dz}z.$$
The integrand then equals
\beq(2-z^2-z^{-2})
\sum_{a=0}^\infty {{a+k-1}\choose k-1}t^az^a
\sum_{b=0}^\infty {{b+k-1}\choose k-1}t^bz^{-b}\label{eq:9}\eeq
times $\frac{dz}z$, and so we need to pick
out the constant term in (\ref{eq:9}). These
terms will occur when $a=b$
and $a=b\pm2$ and so the integral equals:
\beq\sum\left[{a+k-1\choose k-1}^2-{a+k\choose k-1}{a+k-2\choose k-1}
\right]t^{2a}.\eeq
The term in brackets simplifies to $\frac1{a+k-1}{a+k-1\choose k-2}
{a+k-1\choose k-1}$.  For fixed $k$ this is a polynomial in $a$ of degree
$2k-4$, so the series equals $(1-t^2)^{3-2k}N_k(t^2)$ where $N_k$
is a polynomial.  With some help from
Maple we now prove:
\begin{multline}
\sum\frac1{a+k-1}{a+k-1\choose k-2}
{a+k-1\choose k-1}t^{a}=\\(1-t)^{3-2k}\sum\frac1{k-2}
{k-2\choose i}{k-2\choose i+1}t^i\end{multline}
for $k\ge3$.
Using Zeilberger's algorithm on Maple we
see that each side of the equation satisfies
$$(kt^2-2kt+t^2+k-2t+1)X(k+2)+(-2kt-2k
+t+1)X(k+1)+(k-2)X(k)=0.$$
That the two sides agree for the initial
values $k=3,4$ follows from the computations
following Corollary~\ref{cor:2.3}.
\begin{thm} For $k\ge3$ the Poincar\'e series of
$\bar{R}(2,k)$ equals
$$(1-t)^{-2k}(1-t^2)^{3-2k}\sum\frac1{k-2}
{k-2\choose i}{k-2\choose i+1}t^{2i}$$
\label{thm:3.2}\end{thm}
\begin{remark} This fraction is in lowest terms,
since the numerator is positive if $t=\pm1$.
If $k$ is even the numerator has a factor of
$1+t^2$.

\end{remark}
\section{Three-By-Three Matrices}
We now turn to the three-by-three case.  It will simplify
the computations if instead of taking each $|z_i|$ to be 1,
we instead take 
$$t\ll R_1=|z_1|<R_2=|z_2|<R_3=|z_3|.$$
The justification for doing this is based on Weyl's integration formula.
If $s_\lambda(x_1,\ldots,x_n)$ and $s_\mu(x_1,\ldots,x_n)$ are
Schur functions, then their inner product, which equals $\delta_{\lambda,
\mu}$ can be computed as the coefficient of~1 in
$$(n!)^{-1}s_\lambda(x_1,\ldots,x_n)s_\mu(x_1^{-1},\ldots,x_n^{-1})
\prod_{i\ne j}(1-\frac{z_i}{z_j}).$$
Traditionally, this coefficient is computed by integrating over the
torus $|z_i|=1$ for all~$i$, but it can be computed equally well using
any torus $|z_i|=R_i$ and our choice of torus will simplify the computation.
Then the Poincar\'e series of $\bar{C}$ for $k$ $3\times3$
matrices equals $\frac16(2\pi i)^{-3}(1-t)^{-3k}$ times the integral over
$|z_1|=R_1,$ $|z_2|=R_2$, $|z_3|=R_3$ of
$$
\frac{(1-\frac{z_1}{z_2})(1-\frac{z_1}{z_3})(1-\frac{z_2}{z_1})
(1-\frac{z_2}{z_3})(1-\frac{z_3}{z_1})(1-\frac{z_3}{z_2})}
{(1-t\frac{z_1}{z_2})^k(1-t\frac{z_1}{z_3})^k(1-t\frac{z_2}{z_1})^k
(1-t\frac{z_2}{z_3})^k(1-t\frac{z_3}{z_1})^k(1-t\frac{z_3}{z_2})^k}
$$
times $\frac{dz_1}{z_1}\wedge\frac{dz_2}{z_2}\wedge\frac{dz_3}{z_3}$  Clearing fractions gives
$(z_1z_2z_3)^{2k-3}$ times the integral of
\begin{equation}
\frac{-(z_1-z_2)^2(z_1-z_3)^2(z_2-z_3)^2 dz_1\wedge dz_2
\wedge dz_3}{
(z_1-tz_2)^k(z_1-tz_3)^k(z_2-tz_1)^k(z_2-tz_3)^k
(z_3-tz_1)^k(z_3-tz_1)^k}\label{eq:2}
\end{equation}
The poles for $z_1$ are at $tz_2$, $tz_3$, $z_2/t$ and $z_3/t$.
However, only the first two have absolute value less than 
or equal to $R_1$, so only these contribute to the integral.
The residue at $z_1=tz_2$ is gotten by cancelling one 
factor of $(2\pi i)^{-1}$ and the $(z_1-tz_2)^k$, taking the $(k-1)$-st derivative with respect to $z_1$, divinding by $(k-1)!$, and then substituting
$z_1=tz_2$.  Let $D_1$ be the partial derivative with respect
to $z_1$.  Then the derivative in question will 
$(z_2z_3)^{2k-3}(z_2-z_3)^2(z_2-tz_3)^{-k}(z_3-tz_2)^{-k}$
times the sum over $a+\cdots+f=k-1$ of $-{k-1\choose a,b,c,d,e,f}$ times
\begin{multline*}
D_1^a(z_1^{2k-3})D_1^b(z_1-z_2)^2D_1^c(z_1-z_3)^2\\
D_1^d(z_1-tz_3)^{-k}D_1^e(z_2-tz_1)^{-k}D_1^f(z_3-tz_1)^{-k}
\end{multline*}
This summand equals a constant times a power of $z_1$ and
a power of $t$ times
$$(z_1-z_2)^{2-b}(z_1-z_3)^{2-c}(z_1-tz_3)^{-k-d}
(z_2-tz_1)^{-k-e}(z_3-tz_1)^{-k-f}.$$
If we are only concerned with the denominator we can 
investigate only terms with $a=0$, because the other derivatives increase the degrees of terms in the denominator
and the derivatives of $z_1^{2k-3}$ do not.  Ignoring
this term and specializing $z_1=tz_2$ we get the sum
of terms
$$(z_2t-z_2)^{2-b}(z_2t-z_3)^{2-c}(z_2t-tz_3)^{-k-d}
(z_2-t^2z_2)^{-k-e}(z_3-t^2z_2)^{-k-f}$$
or, $\pm(z_2)^{2-b-k-e}t^{-k-e}$ times
\begin{equation}\label{eq:2a}
(1-t)^{2-b}(1-t^2)^{-k-e}(z_3-z_2t)^{2-c}(z_2-z_3)^{-k-d}
(z_3-t^2t_2)^{-k-f}\end{equation}
all times $(z_2-z_3)^2(z_2-tz_3)^{-k}(z_3-z_2t)^{-k}$.
The only pole for $z_2$ with absolute value less than~$R_2$
is $z_2=tz_3$, and the order of the pole is~$k$.  For a given
$a,\ldots,f$ the residue will be a constant times a power of $z_3$ times 
a power of $t$ times $(1-t)^{2-b}(1-t^2)^{-k-e}$ times
$$(1-t^2)^{2-c-g}(1-t)^{-k-d-h}
(1-t^3)^{-k-f-i}(1-t)^{2-j}(1-t^2)^{-k-l}
$$
where $g+h+i+j+l=k-1$.  The pole at $t=-1$ has order
$k+e-2+c+g+k+l$.  Since $c+e\le k-1$ and $g+l\le k-1$
this is at most $2k-2+2(k-1)=4k-4$.  The pole at
$t=\sqrt[3]{1}$ has order $k+f+i$ which is maximized by
taking $f=i=k-1$ in which case it equals $3k-2$.
\begin{lem} Taking the residue of (\ref{eq:2}) with 
respect to $z_1=tz_2$ gives a fraction with denominator
 a power of $(1-t)$ times $(1-t^2)^{4k-4}(1-t^3)^{3k-2}
$.\end{lem}
We next turn to the residue of (\ref{eq:2}) at $z_1=tz_3$.
The first steps in the computation are identical to those
of the residue at $z_1=tz_2$, with $z_2$ and $z_3$
switched.  Analogous to (\ref{eq:2a}) we get that
the denominator is the same as that of
 the sum
of terms $(z_2-z_3)^2(z_2-tz_3)^{-k}(z_3-tz_2)^{-k}$
times
$$(z_3t-z_3)^{2-b}(z_3t-z_2)^{2-c}(z_3t-tz_2)^{-k-d}
(z_3-t^2z_3)^{-k-e}(z_2-t^2z_3)^{-k-f}$$
where again $b+\cdots+f\le k-1$, or, $\pm(z_3)^{2-b-k-e}t^{-k-e}$ times
$$
(1-t)^{2-b}(1-t^2)^{-k-e}(z_2-z_3t)^{2-c}(z_3-z_2)^{-k-d}
(z_2-t^2z_3)^{-k-f}$$
all times $(z_3-z_2)^2(z_3-tz_2)^{-k}(z_2-z_3t)^{-k}$. 
Now there are two poles for $z_2$ inside of $|z_2|=R_2$:
namely $z_2=tz_3$ and $z_2=t^2z_3$.
In the former case the order of the pole is $k-2+c$ and
there will be no  factors
of $(1-t^3)$.  The factors of $(1-t^2)$ will come from
the terms
$$(1-t^2)^{-k-e}(z_3-tz_2)^{-k}$$
and that differentiation will increase the order of the
factor by as much as $k-3+c$.  Keeping in mind that
$b+e\le k-1$ the order will be
$$k+e+k+k-3+a\le3k-3+a+e$$  Since $a+e\le k-1$ this implies the following:
\begin{lem} The residue of (\ref{eq:2}) at $z_1=tz_3$
and then at $z_2=tz_3$ is a rational function with
denominator a power of $(1-t)$ times $(1-t^2)^{4k-4}$
\end{lem}
We leave the case of $z_2=t^2z_3$ to the reader.  The 
result is
\begin{lem} The residue of (\ref{eq:2}) at $z_1=tz_3$
and then at $z_2=t^2z_3$ is a rational function with
denominator a power of $(1-t)$ times $(1-t^2)^{4k-4}
(1-t^3)^{3k-2}$.
\end{lem}
Combining the three lemmas we know have
\begin{thm} The Poincar\'e series of $\bar{C}(3,k)$
is a rational function with denominator 
$(1-t)^{2k-2}(1-t^2)^{4k-4}(1-t^3)^{3k-2}$.
\label{th:4.4}\end{thm}
\begin{proof} The powers of $(1-t^2)$ and
$(1-t^3)$ follow from the lemmas and the power of
$(1-t)$ follows from the fact that the order of the pole
at $t=1$ is $9k-8$ by Lemma~\ref{lem:1.1}.
\end{proof} 
Taking into account Lemma~\ref{lem:1.2} we can say this
about the numerator:
\begin{cor} If $P(\bar{C(3,k)})$ is written as a fraction
with denominator as above, then the numerator is
a symmetric polynomial of degree $10k-16$.
\end{cor}

Doron Zeilberger \cite{Z} computed $P(\bar{C}
(3,k))$ for $k\le100$.  In each case our 
denominator appears to be least and we conjecture that it always is.  Here are
the first few numerators:
\begin{align*}
N_2(t)&=1-t^2+t^4\\
N_3(t)&=t^{14}-t
^{13}-2t^{12}+6t^{11}+6t^{10}-9t^9+t^8+17t^7\\&\quad+
t^6-9t^5+6t^4+6t^3-2*t^2-t+1\\
N_4(t)&=t^{24}-2t^{23}-t^{22}+18t^{21}+6t^{20}-30t^{19}\\&\quad+75t^{18}+150
t^{17}-30t^{16}+30t^{15}+401t^{14}+238t^{13}
\\&\quad
-76t^{12}+238t^{11}+401t^{10}+30t^9-30t^8+
150t^7\\&\quad+75t^6-30t^5+6t^4+18t^3-t^2-2t+1
\end{align*}

The computation of the denominator for the Poincar\'e
series of $\bar{R}(3,k)$ can be carried out similarly.
In this case the numerator of the integrand has an
additional factor $\sum\frac{z_i}{z_j}$ which can
also be written $\sum z_i\sum z_i^{-1}$.  Again, there
are three pairs of residues to consider.  If we first take
$z_1=tz_2$ and then $z_2=tz_3$ then in the terms 
we consider, the extra term in the numerator will 
become
$$(1+t+t^2)(1+t^{-1}+t^{-2})=t^{-2}(1+t+t^2)^2.$$
Hence, in the denominator the power of $(1-t^3)$ will
be $3k-4$ instead of $3k-2$.

For the next case we take the residue at $z_1=tz_3$ and
then at $z_2=tz_3$.  The extra term in the numerator
contributes $(z_3+2tz_3)^2$, which has no effect
on the denominator.  So there are no factors of $(1-t^3)$ and $4k-4$ factors of $(1-t^2)$,
as before.

Finally, if we first take $z_1=tz_3$ and then $z_2=t^2z_3$
we again pick up two factors of $(1+t+t^2)$ in the numerator,
cancelling two in the denominator.  Altogether we now have
\begin{thm} The Poincar\'e series of $\bar{R}(3,k)$
is a rational function with denominator 
$(1-t)^{2k}(1-t^2)^{4k-4}(1-t^3)^{3k-4}$ and numerator
a symmetric polynomial 
of degree $10k-20$.
\end{thm}
\section{A Conjecture}
In this section we present a conjecture for the
denominator for $\bar{C}(n,k)$ and $\bar{R}(n,k)$.
The computations are based on the following function:
\beq\label{eq:14}
F(t)=\frac{\prod\{1-t^{|i-j|}|1\le i,j\le n,\ i\ne j\}}
{\prod\{(1-t^{|i-j+1|})^k|1\le i,j\le n,\ i\ne j-1\}}\eeq
We leave the proof of the following to the interested
reader.
\begin{lem} $F(t)$ in the above equation equals
$\prod_{i=1}^n(1-t^i)^{-\alpha(i)}$, where
$$\alpha(i)=\begin{cases} 2(k-1)(n-i),&1\le i\le n-1\\
k,&i=n
\end{cases}$$\label{lem:5.1}
\end{lem}
Here then is our conjecture for the denominator of 
$\bar{C}(k,n)$:
\begin{conj} $\bar{C}(n,k)$ is a rational function
with least denominator the same as the denominator
of the $(n-1)(k-1)$-st derivative of $F(t)$.
\end{conj}
This denominator can be expressed more explicitly in
two different ways.  One is that it is the least common
multiple of $\{(1-t^i)^{\alpha(i)+(n-1)(k-1)}\}$.  For
the other, let $\phi_i(t)$ be the $i$-th cyclotomic
polynomial, i.e., the minimal polynomial of a primitive
$i$-th root of~1.  It is known that $1-t^i$ is
the product of the $\phi_j(t)$ for $j|i$, and so the
denominator of $F(t)$ is a product of the $\phi_i(t)$, $i\le n$,
to various powers.   Then it follows from the product rule
that the conjectured
denominator for $\bar{C}(n,k)$ will be
the denominator of $F(t)$ times $(\phi_1(t)\cdots\phi_n(t)
)^{(n-1)(k-1)}$.  

We stumbled across this denominator
when attempting the substitutions $z_1=tz_2,\ z_2=
tz_3,\ldots,z_{n-1}=tz_n$ in (\ref{eq:3a}), but the
real evidence in its favor is that it agrees with
the known denominators for $n=2$ and $n=3$
computed here, and for $n=4,5,6$ and $k=2$
computed in~\cite{D}.  

Let us check the $n=3$
case.  The rational function $F(t)$ would be the
reciprocal of
$$(1-t)^{4(k-1)}(1-t^2)^{2(k-1)}(1-t^3)^{k}.$$
The order of the poles of $F(t)$ is as follows:  At
$t=1$ the order is $7k-6$, at $t=-1$ the order
is $2k-2$, and at $\sqrt[3]{1}$ the order is $k$.  Taking the $2(k-1)$-st
derivative increases the order of each pole by $2(k-1)$,
giving orders of $9k-8$, $4k-4$,  and
$3k-2$, respectively.  This is in agreement with
Theorem~\ref{th:4.4}.  

We record the denominators
for $\bar{C}(n,2)$ for $n=4,5,6$ computed by 
Dokovic in~\cite{D}, and which he denoted $D(C_{n,2};t)$, should
the reader wish to check them.
\begin{align*}
D(C_{4,2};t)&=(1-t)^3(1-t^2)^4(1-t^3)^5(1-t^4)^5\\
D(C_{5,2};t)&=(1-t^2)^6(1-t^3)^8(1-t^4)^6(1-t^5)^6\\
D(C_{6,2};t)&=(1`-t)^5(1-t^2)^3(1-t^3)^6(1-t^4)^9(1-t^5)^7(1-t^6)^7
\end{align*}

The case of $\bar{R}(n,k)$ is similar.  The definition of $F(t)$
from (\ref{eq:14}) gains a factor of $\sum z^{i-j}=\sum
z_i\sum z_j^{-1}$ in the numerator, which, up to a
power of $z$, is $(\sum z_i)^2$
suggesting
\beq
G(t)=\frac{(\sum_{i,j=1}^n t^i)^2\prod\{1-t^{i-j}|1\le i,j\le n,\ i\ne j\}}
{\prod\{(1-t^{i-j+1})^k|1\le i,j\le n,\ i\ne j-1\}}\eeq
The extra factor in the numerator equals $(1+t+\cdots+t^{n-1})^2$.
Hence, if $G(t)=\prod (1-t^i)^{-\beta(i)}$, then $G(t)$ has two fewer
factors of $(1-t^n)$ and two more factors of $(1-t)$ than $F(t)$,
namely
$$\beta(i)=\begin{cases}(2(k-1)(n-1)+2,&i=1\\
2(k-1)(n-i),&2\le i\le n-1\\
k-2,&i=n
\end{cases}$$
Here is our conjecture for the denominator of $\bar{R}(n,k)$.
\begin{conj} $\bar{R}(n,k)$ is a rational function
with least denominator the same as the denominator
of the $(n-1)(k-1)$-st derivative of $G(t)$.
\end{conj}

Just as in the previous case the conjecture holds for all known cases.
And, just as in the previous case, this denominator can be computed
either as the least common multiple of $\{(1-t)^{\beta(i)+(n-1)(k-1)}\}$
or as the denominator of $G(t)$ times $(\phi(1)\cdots\phi(n))^{(n-1)(k-1)}$
Here are the denominators computed by Dokovi\'c in~\cite{D} for $n=4,5,6$.
\begin{align*}
D(T_{4,2};t)&=(1-t)^5(1-t^2)^4(1-t^3)^5(1-t^4)^3\\
D(T_{5,2};t)&=(1-t)^2(1-t^2)^6(1-t^3)^8(1-t^4)^6(1-t^5)^4\\
D(T_{6,2};t)&=(1`-t)^7(1-t^2)^3(1-t^3)^6(1-t^4)^9(1-t^5)^7(1-t^6)^5
\end{align*}

If these conjectures are true they have an application to the
growth functions of $\bar{C}(n,k)$ and $\bar{R}(n,k)$.  In
order to apply them need this lemma.
\begin{lem} Let $f(x)$ and $g(x)$ be monic polynomials with
no factors of $(1-x)$, let 
$$\frac{f(x)}{(1-x)^dg(x)}=\sum a_nx^n,$$
and assume that all of the poles of $g(x)$ are on the
unit circle and are of order less than $d$.
Then $a_n$ is asymptotic to $\frac{f(1)}{g(1)}{n+d-1\choose d-1}$ or $\frac{f(1)}{(d-1)!g(1)}n^{d-1}$\label{lem:5.2}
\end{lem}
\begin{proof} By partial fractions
$$\frac{f(x)}{(1-x)^dg(x)}=\frac A{(1-x)^d}+ 
\sum\frac{B_{\omega,i}}{(1-\omega x)^i}$$
where each $\omega$ is a root of 1 and where each $i<d$.
Multiplying both sides by $(1-x)^d$ yields
$$\frac{f(x)}{g(x)}=A+\mbox{ terms with factors of }(1-x)$$
and setting $x=1$ gives $A=f(1)/g(1)$.  The Taylor series
of $(1-x)^{-d}$ is $\sum {n+d-1\choose d-1}x^d$ and the
Tayor series of each $(1-\omega x)^{-i}$ has coefficients
polynomial of degree $i-1$ which is less than $d-1$.
\end{proof}
The algebra $\bar{C}(n,k)$ is graded by degree.  If we let
$\bar{c}_i$ be the dimension of the degree $i$ part then
$P(\bar{C}(n,k))=\sum \bar{c}_it^i$, by definition of
Poincar\'e series. It follows from Lemma~\ref{lem:5.2}
that $\bar{c}_i$ is asymptotic to a rational number times
$i$ to the power of $(k-1)n^2$.  Moreover, since the
generating function for $\sum_{j=1}^i \bar{c}_j$ is
$P(\bar{C}(n,k))(1-t)^{-1}$, the sums $\sum\bar{c}_j$
are asymptotic to a rational number times $i$ to the power
of $(k-1)n^2+1$.  If Conjecture~1 is true, then we can
identify a denominator for those rational numbers.
\begin{thm} If Conjecture 1 is true then $\bar{c}_i$ is asymptotic
to a rational number times $i$ to the power of $(k-1)n^2$,
and the rational number can be written with denominator
$$[n^2(k-1)]!\prod_{j=1}^n j^{\alpha(j)}(\phi_1\cdots \phi_n)^{(k-1)
(n-1)},$$
moreover $\sum_{m=1}^i \bar{c}_m$ is asymptotic to a rational number times $i$ to the power of $(k-1)n^2+1$,
and the rational number can be written with denominator
$$[n^2(k-1)+1]!\prod_{j=1}^n j^{\alpha(j)}(\phi_1\cdots \phi_n)^{(k-1)
(n-1)}$$
\end{thm}

Likewise, if we let $\bar{r}_i$ be the dimension of the degree $i$ 
part of $\bar{R}(n,k)$ then we have this theorem. 
\begin{thm} If the Conjecture 2 is true then $\bar{r}_i$ is asymptotic
to a rational number times $i$ to the power of $(k-1)n^2$,
and the rational number can be written with denominator
$$[n^2(k-1)]!\prod_{j=1}^n j^{\beta(j)}(\phi_1\cdots \phi_n)^{(k-1)
(n-1)},$$
moreover $\sum_{m=1}^i \bar{r}_m$ is asymptotic to a rational number times $i$ to the power of $(k-1)n^2+1$,
and the rational number can be written with denominator
$$[n^2(k-1)+1]!\prod_{j=1}^n j^{\beta(j)}(\phi_1\cdots \phi_n)^{(k-1)
(n-1)}$$
\end{thm}

\end{document}